\documentclass{llncs}
\usepackage{apacite}
\usepackage{amssymb}
\usepackage{amsfonts}
\usepackage{amsmath}
\usepackage{comment}
\usepackage{multirow}
\usepackage{longtable}
\usepackage{chngpage}
\usepackage{color}

\def\Z{\mathbb{Z}}

\def\row#1#2{{#1}_1,\ldots ,{#1}_{#2}}

\begin{document}
\mainmatter  

\title{Hierarchical Simple Games:  Weightedness and Structural Characterization
}
\titlerunning{Hierarchical Simple Games: Representations and Weightedness}

\author{\bf Tatiana Gvozdeva, 
Ali Hameed 
and Arkadii Slinko
}

\institute{Department of Mathematics, The University of Auckland, Private Bag 92019, Auckland, New Zealand\newline
\email{\{t.gvozdeva,a.hameed,a.slinko\}@auckland.ac.nz}
}


%
%

\maketitle

\begin{abstract}
In many situations, both in human and artificial societies, cooperating agents have different  status with respect to the activity and it is not uncommon that certain actions are only allowed to coalitions that  satisfy certain criteria, e.g., to sufficiently large coalitions or coalitions which involve players of sufficient seniority. Simmons (1988)  formalized this idea in the context of secret sharing schemes by defining the concept of a (disjunctive) hierarchical access structure. Tassa (2007) introduced their conjunctive counterpart. From the game theory perspective access structures in secret sharing schemes are simple games.

 In this paper we prove the duality between disjunctive and conjunctive hierarchical games. We introduce a canonical representation theorem for both types of hierarchical games and characterize disjunctive ones as complete games with a unique shift-maximal losing coalition. We give a short combinatorial proof of the Beimel-Tassa-Weinreb (2008) characterization  of weighted disjunctive hierarchical games. By duality we get similar theorems for conjunctive hierarchical games.
\end{abstract}

\section{Introduction}\label{sec:intro}
In many situations cooperating agents  have different  status with respect to the activity.  In the theory of simple games developed by \cite{vNM:b:theoryofgames} 
seniority of players is modeled by giving them different weights. Such situation arise, for example,  in the context of corporate voting when different shareholders have different number of shares.  The access structure in secret sharing schemes \cite{Simmons:1988,Stinson:1992} can also be modeled by a simple game, but in this theory a different approach in defining seniority is often used.  To this end \cite{Simmons:1988} introduced the concept of a hierarchical access structure. Such an access structure stipulates that agents are partitioned into $m$ levels, and a sequence of thresholds $k_1<k_2<\ldots<k_m$ is set, such that a coalition is authorized if it has either $k_1$ agents of the first level or $k_2$ agents of the first two levels or $k_3$ agents of the first three levels etc.  Consider, for example, the situation of a money transfer from one bank to another. If the sum to be transferred is sufficiently large, this transaction must be authorized by three senior tellers or two vice-presidents. However, two senior tellers and a vice-president can also authorize the transaction.  These hierarchical structures are called {\em disjunctive}, since only one of the $m$ conditions must be satisfied for a coalition to be authorized. If all conditions must be satisfied, then the hierarchical access structure is called {\em conjunctive}.  A typical example of a conjunctive hierarchical game would be the United Nations Security Council where for the passage of a resolution all five permanent members must vote for it {\em and} also at least nine members in total.  

It has been shown that these two approaches are seldom  equivalent since hierarchical access structures are seldom weighted. Both  \cite{beimel:360} and  \cite{padro:2010} characterized weighted disjunctive hierarchical access structures as a part of their characterization of weighted ideal access structures. They showed that, beyond two levels, disjunctive hierarchical structures are normally non-weighted.  This is extremely interesting from the game-theoretic point of view, since we now have a natural class of non-weighted access structures and hence simple games. However, the proof of this characterization in both papers was indirect. They used the fact that hierarchical access structures are ideal \cite{Brickell:1990:ISS:111563.111607} and the well-known relation between ideal secret sharing schemes and matroids \cite{springerlink:10.1007/0-387-34805-0-25}.  Conjunctive hierarchical access structures, which were introduced in \cite{Tassa:2007}, have got much less attention. We will use the game-theory methods and terminology, and we will talk about hierarchical games, not access structures. 

Progress in studying hierarchical games was hindered by the absence of any canonical representation, which is needed since different values of parameters can give us the same game. 
In this paper we introduce a canonical representation  of hierarchical games, and  give a short combinatorial proof of the Beimel-Tassa-Weinreb characterization theorem by using the technique of trading transforms developed in \cite{tz:b:simplegames}. Our statement is slightly more general, as it allows for the existence of dummy players. We also characterize disjunctive hierarchical games as complete games with a unique shift-maximal losing coalitions.  Then we prove the duality between disjunctive and conjunctive games. This allows us to characterize weighted conjunctive hierarchical games and obtain their structural characterization as complete games with a unique shift-minimal winning coalition.  The class of complete games with a unique shift-minimal winning coalition was studied in its own right  in \cite{RePEc:eee:ejores:v:188:y:2008:i:2:p:555-568}. However, they did not notice that the games which they study are hierarchical conjunctive games.

\section{Preliminaries}

The background material on simple games can be found in \cite{tz:b:simplegames}.

 \begin{definition}
Let $P=[n]=\{1,2,\ldots, n\}$ be a  set of players and let $\emptyset\ne W \subseteq 2^P$ be a collection of subsets of $P$ that satisfies the following property:
 \begin{equation}
 \label{monot}
\text{ if $X\in W$ and $X\subseteq Y$, then $Y\in W$.}
 \end{equation}
 In such case the pair $G=(P,W)$ is called a {\em simple game} and the set $W$ is called the set of {\em winning coalitions} of $G$. Coalitions that are not in $W$ are called {\em losing}.
\end{definition} 

Due to the monotonic property (\ref{monot}) the subset $W $ is completely determined by the set $W_{\text{min}} $ of  minimal winning coalitions of $G$.  A player who does not belong to any minimal winning coalition is called a {\em dummy}. Such a player can be removed from any winning coalition without making it losing.

\begin{definition}
A simple game $G=(P,W) $ is called a {\em weighted majority game}  if there exist nonnegative weights $\row wn$ and a threshold $q$ such that 
\begin{equation}
\label{WMG}
X\in W \Longleftrightarrow \sum_{i\in X}w_i\ge q.
\end{equation}
\end{definition}
In secret sharing, weighted threshold access structures  were introduced by \cite{shamir:1979}.\par\smallskip

A distinctive feature of many games is that the set of players is partitioned into subsets, and players in each of the subsets have equal status.  We suggest  analyzing such games with the help of multisets.  Given a simple game $G$ we define a relation $\sim_{G} $ on $P$ by setting $i \sim_{G} j$ if for every set $X\subseteq P$ not containing $i$ and~$j$ 
\begin{equation}
\label{condition}
X\cup \{i\}\in W \Longleftrightarrow X\cup \{j\} \in W.
\end{equation}

\begin{lemma}
$\sim_{G} $ is an equivalence relation.
\end{lemma}

\begin{example}
Suppose we have $P = \{a_1,a_2,b_1,c_1\}$ as the full set of players with weights as follows: $a_1$ and $a_2$ have weights 1, $b_1$ has weight 2 and $c_1$ has weight 3. Then the following is the set of minimal winning coalitions for the game with $q = 3$. 
\[
W_{min} = \{\{a_1,b_1\},\{a_2,b_1\},\{c_1\}\}.
\]
This gives $a_1 \sim_{G} a_2$ and of course $a_2 \sim_{G} a_1$ as $\sim$ is symmetric. Since  $\sim$ is reflexive, then $a_i\sim a_i$ for $i=1,2$, and also $b_1 \sim_{G} b_1$. Similarly $c_1 \sim_{G} c_1$. It follows that our equivalence classes are $\{a_1,a_2\},\{b_1\}$ and $\{c_1\}$.   
\end{example}

We need now the notion of a {\em multiset}. 

\begin{definition}
A multiset on the set $[m]$ is a mapping $\mu\colon [m]\to \Z_+$ of $[m]$ into the set of non-negative integers. It is often written in the form
\[
\mu = \{1^{k_1},2^{k_2},\ldots, m^{k_m}\},
\]
where $k_i=\mu(i)$ is called the {\em multiplicity} of $i$ in $\mu$. 
\end{definition}

A multiset $\nu = \{1^{j_1},\ldots, m^{j_m}\}$ is a submultiset of a multiset $\mu = \{1^{k_1},\ldots, m^{k_m}\}$, iff $j_i\le k_i$ for all $i=1,2,\ldots, m$. This is denoted as $\nu \subseteq \mu$.\par\smallskip

The existence of large equivalence classes relative to $\sim_G$ allows us to compress the information about the game. This is done by the following construction.
Let now $G=(P,W)$ be a game and $\sim_{G}$ be its corresponding equivalence relation. Then $P$ can be partitioned into a finite number of equivalence classes  $P=P_1\cup P_2\cup\ldots\cup P_m$ relative to $\sim_{G}$ and suppose that $|P_i|=n_i$. Then we put in correspondence to the set of agents $P$ a multiset $\bar{P}=\{1^{n_1},2^{n_2},\ldots, m^{n_m}\}$. We take our base set $P$, identify those agents which are equivalent  and we do not distinguish between them any further. We carry over the game structure to $\bar{P}$ as well by defining the set of submultisets $\bar{W}\subseteq \bar{P}$ by assuming that a submultiset  $Q=\{1^{\ell_1},2^{\ell_2},\ldots, m^{\ell_m}\}$ is winning in $\bar{G}$ if a subset of $P$ containing  $\ell_i$ agents from $P_i$ ($i=1,2,\ldots, m$), is winning in~$G $. This definition is correct since the sets $P_i$ are defined in such a way that it does not matter which $\ell_i$ players from $P_i$ are involved. We will call $\bar{G}=(\bar{P},\bar{W})$ the {\em canonical representation} of $G$.

\begin{definition}
A pair $\bar{G}=(\bar{P},\bar{W})$ where $\bar{P} = \{1^{n_1},2^{n_2},\ldots, m^{n_m}\}$ and  $\bar{W}$ is a  system of submultisets of the multiset  $\bar{P}$ is said to be a {\em simple game} on $\bar{P}$  if $X\in \bar{W} $ and $X\subseteq Y$ implies $Y\in\bar{W}$.  
\end{definition}

So the canonical representation of a simple game on a set of players $P$ is a simple game on  the multiset $\bar{P}$. We will omit bars when this does not invite confusion.\par\smallskip

Given a game $G$ on a set of players $P$ we may also define a relation $\succeq_{G} $ on $P$ by setting $i \succeq_{G} j$ if for every set $X\subseteq U$ not containing $i$ and~$j$ 
\begin{equation}
\label{isbel}
X\cup \{j\}\in W \Longrightarrow X\cup \{i\} \in W.
\end{equation}
That is known as Isbel's desirability relation \cite{tz:b:simplegames}. The  game is called {\em complete} if $\succeq_{G}$ is a total (weak) order. We also define the relation $i \succ_{G} j$ as $i \succeq_{G} j$ but not $j \succeq_{G} i$.

\begin{definition}
We say that $\bar{G}=(\bar{P},\bar{W})$ is a weighted majority game if there exist non-negative weights $\row wm$ and $q\ge 0$ such that  $Q=\{1^{\ell_1},2^{\ell_2},\ldots, m^{\ell_m}\}$ is winning  iff ${\sum_{i=1}^m \ell_iw_i\ge q}$.
\end{definition}

If $G$ is weighted, then it is well-known (see, e.g., \cite{tz:b:simplegames}, p.91) that we can find a weighted representation, for which  equivalent players have equal weights. Hence we obtain

\begin{proposition}
A simple game  $G=(P,W)$ is a weighted majority game if and only if the corresponding simple game $\bar{G}=(\bar{P},\bar{W})$ is. 
\end{proposition}

One of the most interesting classes of complete games is hierarchical games. They can be of two types (\cite{beimel:360}, \cite{Tassa:2007}), and they will be considered in the next section. \par\smallskip

If a game $G $ is complete, then we define {\em  shift-minimal} ($\delta$-minimal in \cite{CF:j:complete}) winning coalitions and shift-maximal losing coalitions. By a {\em shift} we mean a replacement of a player  of a coalition by a less desirable player which did not belong to it.  Formally, given a  coalition $X$, player $p\in X$ and another player $q\notin X$ such that $q\prec_{G}p$, we say that the coalition $ (X\setminus \{p\})\cup \{q\} $ is obtained from $X$ by a {\em shift}. A winning coalition $X$ is {\em shift-minimal} if  every coalition contained in it and every coalition obtained from it by a shift are losing. A losing coalition $Y$ is said to be {\em shift-maximal} if every coalition that contains it is winning and there does not exist another losing coalition from which $Y$ can be obtained by a shift.  \par\medskip


The definition of a shift  in the multiset context must be adapted as follows.

\begin{definition}
Let $G$ be a complete simple game on a multiset $P=\{1^{n_1},\ldots, m^{n_m}\}$, where 
$
1\succ_{G}2\succ_{G}\ldots \succ_{G} m.
$
Suppose a submultiset
\[
A'=\{\ldots,i^{\ell_i},\ldots, j^{\ell_j},\ldots \}
\]
has $\ell_i\ge 1$ and $\ell_j<n_j$ for some $i<j$. Then we will say that the submultiset
\[
A'=\{\ldots,i^{\ell_i-1},\ldots, j^{\ell_j+1},\ldots \}
\]
is obtained from $A$ by a shift. 
\end{definition}
Shift-minimal winning and shift-maximal losing coalitions are then defined straightforwardly.

For $X \subset P$ we will
denote its complement $P \setminus X$ by $X^c$. 

\begin{definition}
Let $G=(P,W)$ be a simple game and $A\subseteq P$. Let us define subsets 
\[
W_{\text{sg}}=\{X\subseteq  A^c\mid X\in W\}, \  
W_{\text{rg}}=\{X\subseteq A^c\mid X\cup A\in W\}.
\]
Then the game $G_A=(A^c,W_\text{sg})$ is called a {\em subgame} of $G$ and $G^A=(A^c,W_\text{rg})$ is called a {\em reduced game} of $G$. 
\end{definition}

\begin{proposition}
Every subgame and every reduced game of a weighted majority game is also a weighted majority game.
\end{proposition}


Let us discuss briefly duality in games. The dual game of a
game $G=(P,W)$ is defined as $G^{*} = (P,L^{c})$. Equivalently,
the winning coalitions of the game $G^{*}$ dual to $G$  are
exactly the complements of losing coalitions of $G$. We have $G=G^{**}$.
%
We note also that, If $A \subseteq P$, then:
$
(G_A)^{*} = (G^{*})^A$ and  $(G^A)^{*} = (G^{*})_A.
$
 Moreover, the operation of taking the dual is known to  preserve weightedness. We will also use the fact that Isbel's desirability relation is self-dual, that is $x\succeq_G y$ if and only if 
$x\succeq_{G^{*}} y$.
All these concepts can be immediately reformulated for the games on multisets.\par\medskip


Let us remind the reader of some more facts from the theory of simple games. The sequence of coalitions
\begin{equation}
\label{Tt}
{\cal T}=(\row Xj;\row Yj)
\end{equation}
is called a {\em trading transform} if the coalitions $\row Xj$ can be converted into the coalitions $\row Yj$ by rearranging players. In other words, for any player $p$ the cardinality of the set $\{i\mid p\in X_i\}$ is the same as the cardinality of the set $\{i\mid p\in Y_i\}$. We say that the trading transform $\cal T$ has length $j$. 

\begin{theorem}[\cite{tz:b:simplegames}]
A game $G=(P,W)$ is a weighted majority game if for no $j$ does there exist a trading transform (\ref{Tt}) such that $\row Xj$ are winning and $\row Yj$ are losing.
\end{theorem}

This theorem gives a combinatorial way to prove the existence of weights for a given game.

\begin{definition}
Let $G=(P,W)$ be a simple game. A trading transform (\ref{Tt}) where all $\row Xj$ are winning in $G$ and all $\row Yj$ are losing in $G$ is called {\em certificate of non-weightedness} for~$G$. 
\end{definition}

For complete games the criterion can be made easier to check, by the following result. 

\begin{theorem}[\cite{FMDAM}]
\label{FM}
A complete game is a weighted majority game if and only if it does not have certificates of non-weightedness (\ref{Tt}) such that $\row Xj$ are shift-minimal winning coalitions and $\row Yj$ are losing coalitions.
\end{theorem}

\section{Canonical Representations and Duality of Hierarchical Games}

\begin{definition}[Disjunctive Hierarchical Game] 
\label{HG}
Suppose that the set of players $P$ is partitioned into $m$ disjoint subsets $P=\cup_{i=1}^m P_i$ and let $k_1<k_2<\ldots<k_m$ be a sequence of positive integers. Then we define the game $H=H_{\exists}(P,W)$ by setting
\[
W = \{ X\in 2^P\mid \exists i \left(\left|X\cap \left(\cup_{j=1}^i P_i\right)\right|\ge k_i\right) \}.
\]
\end{definition} 

From the definition it follows that any disjunctive hierarchical game $H$ is complete, moreover for any $i\in [m]$ and $u,v\in P_i$ we have $u\sim_{H}v$. However, for arbitrary values of parameters we cannot guarantee that the canonical representation $\bar{H}$ of $H$ will be defined on the multiset $\bar{P}=\{1^{n_1},2^{n_2},\ldots, m^{n_m}\}$, since it is possible to have less than $m$ equivalence classes.   The next theorem shows when this does not happen.

\begin{theorem}
\label{Thm_1}
Let $H$ be a disjunctive hierarchical game defined on the set of players $P$ partitioned into $m$ disjoint subsets $P=\cup_{i=1}^m P_i$, where $n_i=|P_i|$, by a sequence of positive thresholds $k_1<k_2<\ldots<k_m$.  Then the canonical representation $\bar{H}$ of $H $ has $m$ equivalence classes, and hence it is defined on $\bar{P}= \{1^{n_1},2^{n_2},\ldots, m^{n_m}\}$ if and only if
\begin{enumerate}
\item[(a)] $k_1\le n_1$, and
\item[(b)] $k_i <  k_{i-1}+n_i$  for every $1<i< m$.
\end{enumerate}
When (a) and (b) hold the sequence  $(\row k{m-1})$ is determined uniquely. Moreover, $H$ does not have dummies if and only if $k_m <  k_{m-1}+n_m$; in this case  $k_m$ is determined uniquely as well. If $k_m\ge k_{m-1}+n_m$ the last $m$th level consists entirely of dummies.
\end{theorem}

\begin{proof}
As we know, players within each $P_i$ are equivalent. We note that if $k_1>n_1$, then $P_1\sim_{H}P_2$. On the other hand, if $k_1\le n_1$, then any $k_1$ players from $P_1$ form a winning coalition $M_1$ which ceases to be winning if we replace one of them with a player of $P_2$ yielding $P_1\not\sim_H P_2$. Suppose that we know already that $P_{i-1}\not\sim P_{i}$ for some $i<m$, and that there is a minimal winning coalition $M_{i-1}$ contained in $\cup_{j=1}^{i-1} P_{j}$ which intersects $P_{i-1}$ nontrivially and consists of $k_{i-1}$ players. If  $k_i \ge  k_{i-1}+n_i$, and a coalition $Q\subseteq \cup_{j=1}^i P_j$ is winning and has a nonzero intersection with $P_i$, then we also have $|Q\cap \left(\cup_{j=1}^{i-1} P_{j}\right)|\ge k_{i-1}$ and hence $Q\cap \left(\cup_{j=1}^{i-1} P_{j}\right)$ is also winning. Then  any player of $P_i$ in $Q$ can be replaced by any player of $P_{i+1}$ without $Q$ becoming losing, i.e., $P_i\preceq_H P_{i+1}$. From the definition of hierarchical game we have $P_i\succeq_H P_{i+1}$,  this implies $P_{i}\sim_{H}P_{i+1}$. On the other hand, if $k_i <  k_{i-1}+n_i$, we see that a minimal winning coalition in $\cup_{j=1}^i P_i$ exists which intersects with $P_i$ nontrivially and consists of $k_i$ players. For constructing it we have to take $k_i$ players of the $i$th level (if they are available) and, if their number is less than $k_i$ add $k_i-n_i$  players from $M_{i-1}$. We note that the number of players needed to be added is less than $k_{i-1}$, which makes $M_i$ minimal. As above, the existence of such coalition this implies $P_{i}\not\sim_{H}P_{i+1}$. 

The uniqueness of $(\row k{m-1})$ (and also $k_m$ in case $k_m <  k_{m-1}+n_m$) follows from the fact that these numbers are exactly the cardinalities of minimal winning coalitions in $\bar{H}$.  
\end{proof}

By $H_{\exists}({\bf n},{\bf k})$ we will denote the $m$-level disjunctive hierarchical game canonically represented by  ${\bf n}=(\row nm)$ and ${\bf k}=(\row km)$ with $k_m=k_{m-1}+n_m$ in the case where the last level consists of dummies. Every new level, except maybe the last one, adds a new class of minimal winning coalitions.

\begin{corollary}
\label{n_2>1}
Let $G=H_\exists({\bf n},{\bf k})$ be an $m$-level disjunctive hierarchical game. Then  we have $n_i>1$  for every ${1<i< m}$.
\end{corollary}

\begin{proof}
If $n_i=1$ for some $1<i< m$, then (b) cannot hold.
\end{proof}

We note that the first and the last $m$th level are special. If $k_1=1$, then every user of the first level is self-sufficient (passer) and its presence makes any coalition winning and if $k_m \ge  k_{m-1} +n_m$, then the $m$th level consists entirely of dummies. 

\begin{definition}[Conjunctive Hierarchical  Game] 
\label{CG}
Suppose that the set of agents $P$ is partitioned into $m$ disjoint subsets $P=\cup_{i=1}^m P_i$, and let $k_1< \ldots < k_{m-1}\le k_m$ be a sequence of positive integers. Then we define the game $H_{\forall}(P,W)$ by setting
\[
W = \{ X\in 2^P\mid \forall i \left(\left|X\cap \left(\cup_{j=1}^i P_i\right)\right|\ge k_i\right) \}.
\]
\end{definition} 

\begin{theorem}\label{all-exists-dual}
Let ${\bf n}=(\row nm)$ and ${\bf k}=(\row km)$. Then for an $m$-level hierarchical games $H_\exists({\bf n},{\bf k})^{*}=H_\forall({\bf n}, {\bf k}^{*})$ and $H_\forall({\bf n},{\bf k})^{*}=H_\exists({\bf n}, {\bf k}^{*})$,
where
\[
{\bf k}^{*}=(n_1-k_1+1, n_1+n_2 - k_2+1, \ldots, \sum_{i \in [m]} n_i -k_m +1).
\]
\end{theorem}

\begin{proof}
We will prove only the first equality. 
As Isbel's desirability relation is self-dual, the canonical representation of $H_\exists({\bf n},{\bf k})^{*}$ will involve the same equivalence classes and hence it will be defined on the same multiset.
Let ${\bf k}^{*}=(k^{*}_1, k^{*}_2,\ldots, k^{*}_m)$. It is easy to see that $k^{*}_i<k^{*}_{i+1}$ is equivalent to $k_{i+1}<k_i+n_{i+1}$ so we have  $k^{*}_1<\ldots< k^{*}_{m-1}\le  k^{*}_m$ and $k^{*}_{m-1}=  k^{*}_m$
if and only if $k_m=k_{m-1}+n_m$. So ${\bf k}^{*}$ is well-defined.
Consider a losing in coalition $X =
\{1^{\ell_1}, 2^{\ell_2}, \ldots, m^{\ell_m}\}$ in $H_\exists({\bf n},{\bf k})$. It satisfies
$\sum_{j \in [i]}\ell_j< k_i$ for all $i \in [m]$. Then 
\[
\sum_{j \in [i]}(n_j -\ell_j) > \sum_{j \in [i]}n_j- k_i,
\] 
for all $i \in [m]$, and the coalition $X^{c} =
\{1^{n_1-\ell_1}, 2^{n_2-\ell_2}, \ldots, m^{n_m-\ell_m}\}$
satisfies the  condition $\sum_{j \in [i]}(n_j -\ell_j) \geq
\sum_{j \in [i]}n_j - k_i+1=k^{*}_i,$ for all $i \in [m]$. Therefore, $X^c
$ is winning in $H_\forall ({\bf n}, {\bf k}^{*})$.

We need also to show that the complement of every winning in
$H_\exists({\bf n},{\bf k})$ coalition is losing in
$H_\forall({\bf n},{\bf k}^{*})$. Consider a coalition $X =
\{1^{\ell_1}, 2^{\ell_2}, \ldots, m^{\ell_m}\}$ which is winning
in $H_\exists({\bf n},{\bf k})$. It means that there is an $i \in [m]$ such that $\sum_{j
\in [i]} \ell_j \geq k_i$. But then the condition
\[
\sum_{j \in [i]}(n_j -\ell_j) \leq \sum_{j \in [i]}n_j - k_i < \sum_{j \in [i]}n_j - k_i +1 = k^{*}_i
\] 
holds. Thus, the complement $X^{c}
= \{1^{n_1-\ell_1}, 2^{n_2-\ell_2}, \ldots, m^{n_m-\ell_m}\}$ is
losing in $H_{\forall}({\bf n},{\bf k}^{*})$.
\end{proof}

We note a certain duality for the second parameter as ${\bf k}^{{*}{*}}={\bf k}$.

\begin{theorem}
\label{Thm_2}
Let $H$ be a conjunctive hierarchical game defined on the set of agents $P$ partitioned into $m$ disjoint subsets $P=\cup_{i=1}^m P_i$, where $n_i=|P_i|$, by a sequence of positive thresholds $k_1<\ldots < k_{m-1} \le k_m$.  Then the canonical representation $\bar{H}$ of $H $ has $m$ equivalence classes and, hence, it is defined on $\bar{P}= \{1^{n_1},2^{n_2},\ldots, m^{n_m}\}$ if and only if
\begin{enumerate}
\item[(a)] $k_1\le n_1$, and
\item[(b)] $k_{i} < k_{i-1}+n_i $  for every $1< i\le m$.
\end{enumerate}
When (a) and (b) hold the sequence  $(\row k{m})$ is determined uniquely. The last $m$th level consists entirely of dummies if and only if $k_{m-1}=k_m$.
\end{theorem}

\begin{proof}
This is a direct consequence of duality and Theorem~\ref{Thm_1}. Indeed we have $k^{*}_i<k^{*}_{i-1}+n_i$ if and only if $k_{i-1}<k_i$ and $k^{*}_1\le n_1$ is equivalent to $k_1>0$, $k^{*}_1>0$ is equivalent to $k_1\le n_1$ and $k^{*}_{i-1}<k^{*}_{i}$ is equivalent to $k_i<k_{i-1}+n_i$. 

To prove the second statement we use duality and the fact that ${\bf k}^{{*}{*}}={\bf k}$.
\end{proof}

We will need the following two propositions.


\begin{proposition}
\label{cut_tail}
Let  ${\bf n}=(\row n{m})$, ${\bf k}=(\row k{m})$ and $G=H_\exists({\bf n},{\bf k})$. If  ${\bf n}'=(\row n{m{-}1})$, ${\bf k}'=(\row k{m{-}1})$, then $H({\bf n}',{\bf k}')$ is a subgame $G_A$ of  $G$  for $A=\{m^{n_m}\}$.
\end{proposition}

\begin{proposition}\label{reduced}
Let  ${\bf n}=(\row n{m})$, ${\bf k}=(\row k{m})$ and $G=H_\forall({\bf n},{\bf k})$.
Suppose $k_1=n_1$, ${\bf n}'= (n_1, \ldots, n_m)$, and ${\bf
k}'=(k_2-k_1, \ldots, k_m-k_1)$. Then $H_\forall({\bf n}',{\bf k}')$ is a reduced
game $G^{A}$, where $A = \{1^{n_1}\}$.
\end{proposition}

\section{Characterizations of Disjunctive Hierarchical Games}

Firstly, we will obtain a structural characterization of hierarchical disjunctive games.

\begin{theorem}
\label{shift-maximal-losing}
The class of disjunctive hierarchical simple games is exactly the class of complete games with a unique shift-maximal losing coalition. 
\end{theorem}

\begin{proof}
Let $G=H_\exists ({\bf n},{\bf k})$ be an $m$-level hierarchical game. If $k_m <  k_{m-1}+n_m$, then the following coalition is a shift-maximal losing one:
\begin{equation}
\label{shift_maximal}
M=\{1^{k_1-1},2^{k_2-k_1},\ldots, m^{k_{m}-k_{m-1}}\}.
\end{equation}
Indeed, for every $i=1,2,\ldots,m$ it has $k_i-1$ players from the first $i$ levels, and so any replacement of a player with more influential one makes it winning. If $k_m \ge  k_{m-1}+n_m$, then it has to be modified as
\begin{equation}
\label{shift_maximal2}
M=\{1^{k_1-1},2^{k_2-k_1},\ldots, (m-1)^{k_{m-1}-k_{m-2}}, m^{n_m}\}.
\end{equation}
Suppose now that $G$ is complete, has a canonical multiset representation on a multiset $P=\{1^{n_1},2^{n_2}, \ldots, m^{n_m}\}$ and has a unique shift-maximal losing coalition $M=\{1^{\ell_1},2^{\ell_2}, \ldots, m^{\ell_m}\}$. We claim that $\ell_i<n_i$ for all $1\le i<m$. Suppose not.  We know that there exists a multiset $X$ such that $X\cup \{i\}$ is winning but $X\cup \{i+1\}$ is losing. We first take $X$ to be of maximal possible cardinality first, and then shift-maximal with this property. This  will make 
$X\cup \{i+1\}$ a shift-maximal losing coalition. Indeed, we cannot add any more elements to $X$, and replacement any element of it with the more influential one makes it winning. Since $X\cup \{i+1\}$ is not equal to $M$ (the multiplicity of $i$ is not at full capacity) we get a contradiction. Hence $\ell_i<n_i$. Then $\{1^{\ell_1},\ldots,  (i-1)^{\ell_{i-1}}, i^{\ell_i+1}\}$ must be winning. Then every coalition with $k_i=\ell_1+\ldots+\ell_i+1$ player from the first $i$ levels is winning. Now if $\ell_m=n_m$ we set $k_m=k_{m-1}+n_m$, alternatively we set $k_m=\ell_1+\ldots+\ell_m+1$. 
It is easy to see that  $G$ is in fact $H_\exists({\bf n},{\bf k})$.
\end{proof}

\cite{beimel:360}  characterized ideal weighted threshold secret-sharing schemes. As part of this characterization they characterized hierarchical weighted games. However, their proof is indirect, and relies heavily upon the connection between ideal secret-sharing schemes and matroids.  Here we will prove the following theorem, which is slightly more general than their Claim 6.5. In secret-sharing dummies are not allowed to be present, so they have at most three levels, not four.

\begin{theorem}
\label{BTW}
Let $G=H_\exists({\bf n},{\bf k}) $ be an $m$-level hierarchical simple game. Then $G $ is a weighted majority game iff one of the following conditions is satisfied:
\begin{enumerate}
\item[(1)] $m=1$;
\item[(2)] $m=2$ and $k_2=k_1+1$;
\item[(3)] $m=2$ and $n_2= k_2-k_1+1$;
\item[(4)] $m\in \{2,3\}$ and $k_1=1$. When $m=3$, $G$ is weighted if and only if the subgame $H_\exists({\bf n}',{\bf k}')$, where ${\bf n}'=(n_2,n_3)$ and ${\bf k}'=(k_2,k_3)$ falls under (2) or (3);
\item[(5)] $m\in \{2,3,4\}$,  $k_m\ge k_{m-1}+n_m$,  and the subgame $H_\exists({\bf n}',{\bf k}')$, where ${\bf n}'=(n_1,\ldots, n_{m-1})$ and ${\bf k}'=(k_1,\ldots, k_{m-1})$ falls under one of the (1) -- (4);
\end{enumerate}
\end{theorem} 

\begin{proof}
We will prove this theorem using the combinatorial technique of trading transforms.
If $k_m\ge  k_{m-1}+n_m$, then users of the last level are dummies and they never participate in any minimal winning coalition. As a result,  if there exists a certificate of non-weightedness
\begin{equation}
\label{Tt1}
{\cal T}=(\row Xj;\row Yj)
\end{equation}
with minimal winning coalitions $\row Xj$, which exist by Theorem~\ref{FM}, then no dummies may be found in any of the  $\row Xj$, hence they are not participating in this certificate. 
Hence $G $ is weighted if and only if its subgame $H_\exists({\bf n}',{\bf k}')$, where ${\bf n}'=(n_1,\ldots, n_{m-1})$ and ${\bf k}'=(k_1,\ldots, k_{m-1})$ is weighted. So we reduce our theorem to the case without dummies, and in this case we have to prove that $G$ falls under the one of the cases (1)-(4). Let us assume that $k_m <   k_{m-1}+n_m$.\par

If $k_1=1$, then every user of the first level is self-sufficient (passer), that is, any coalition with participation of this agent is winning. If a certificate of non-weightedness~(\ref{Tt1}) exists, then a $1$ cannot be a member of any set $\row Xj$, since then it will have to be also in one of the $\row Yj$ and at least one of them will not be losing. Hence $G $ is weighted if and only if its  subgame $H_\exists({\bf n}',{\bf k}')$, where ${\bf n}'=(n_2,\ldots, n_m)$ and ${\bf k}'=(k_2,\ldots, k_m)$  is weighted.

Now we assume $k_1\ge 2$. The case $m=1$ is trivial.  Next we show that if we are restricted to two levels such that condition (5) is not met but any one of the two conditions (2) and (3) is met, then $G $ is weighted. So we assume that  $m = 2$ and $k_1\ge 2$.  If $k_2 \ge k_1+n_2$ we have case (5); so suppose $k_1\le n_1$ and $k_2 < k_1+n_2$. If $k_2=k_1+1$ then this leads to weightedness. Indeed, suppose we have a certificate of non-weightedness~(\ref{Tt1}) with $\row Xj$ winning and $\row Yj$ losing coalitions. We have then $|X_i|\ge k_1$ and $|Y_i|<k_2$ for all $i$. Thus we have $|X_i|=|Y_j|=k_1$ for all $i,j$. Since $|X_i|=k_1$ and winning, it must be $X_i=\{1^{k_1}\}$ for all $i$. But this will imply that  $Y_i=\{1^{k_1}\}$ for all $i$, which is an absurd as $Y_i$ must be losing. \par
 
Now we show that $m=2$ together with $n_2= k_2-k_1+1$ (we note that by Theorem~\ref{Thm_1} this is the smallest value that $n_2$ can take) implies $G $ is a weighted majority game. Assume towards a contradiction that $G $ is not weighted, then there exists a certificate of non-weightedness~(\ref{Tt1}).  If $k_2 = k_1 + 1$,  we know that $G $ is weighted. So assume that $k_2 \geq k_1 + 2$. The first shift-minimal winning coalition is $\{1^{k_1}\}$. As $k_1>1$, it follows that $n_2 < k_2$,  which implies that $\{2^{k_2}\}$ is not a legitimate coalition.  As $k_2-n_2=k_1-1$, the second shift-minimal winning  coalition is therefore $\{1^{k_2-n_2},2^{n_2}\}=\{1^{k_1-1},2^{n_2}\}$. There are no other cases.

In the certificate of non-weightedness (\ref{Tt1}) we may assume that $\row Xj$ are shift-minimal, that is of the two types described earlier.  It is obvious that no $\{1^{k_1}\}$ can be among $\row Xj$. Hence $X_1=\ldots =X_j=\{1^{k_1-1},2^{n_2}\}$. It is now clear that we cannot distribute all ones and twos among $\row Yj$ so that they are all losing.

Conversely, we show that if all conditions (1)-(3) fail, then $G $ is not weighted. If $m=2$, this means that $k_2\ge k_1+2$ and $n_2\ge k_2-k_1+2$. In this case the game possesses  the following certificate of non-weightedness:
\begin{align*}
&(\{1^{k_1}\}, \{1^{{k_1}-2}, 2^{k_2-k_1+2}\};\\  
&\{1^{{k_1}-1}, 2^{\lfloor (k_2-k_1+2)/{2}\rfloor}\}, \{1^{{k_1}-1}, 2^{\lceil (k_2-k_1+2)/{2}\rceil}\}).
\end{align*}
Since $n_2\ge k_2-k_1+2$,  all the coalitions are well-defined. Also, the restriction $k_2\ge k_1+2$ secures that $\lceil \frac{k_2-k_1+2}{2}\rceil \le k_2-k_1$ and makes both multisets in the right-hand-side of the trading transform losing.\par\medskip

Now suppose $m\ge 3$, $k_1 \geq 2$ and the condition (5) is not applicable. We may also assume that  $k_1\le n_1$, $k_2<k_1+n_2$ and $k_3<k_2+n_3$. Suppose first that $k_3\le n_3$. Then, since $k_3\ge k_2+1\ge k_1+2\ge 4$, the following is a certificate of non-weightedness. 
\[
(\{1^{k_1}\},\{3^{k_3}\};\{1^{k_1-1},3^2\},\{1,3^{k_3-2}\}).
\]
Suppose $k_3>n_3$. If at the same time $k_3\le n_2+n_3$, then since $k_3-n_3<  k_2$ we have a legitimate certificate of non-weightedness 
\[
(\{1^{k_1}\},\{2^{k_3-n_3},3^{n_3}\};\{1^{k_1-1},2,3\},\{1,2^{k_3-n_3-1},3^{n_3-1}\}).
\]
Finally, if $k_3>n_3$ and $k_3 > n_2+n_3$, then the certificate of non-weightedness will be
\begin{align*}
&(\{1^{k_1}\},\{1^{k_3-n_2-n_3},2^{n_2},3^{n_3}\};\\
&\{1^{k_1-1},2,3\},\{1^{k_3-n_2-n_3+1},2^{n_2-1},3^{n_3-1}\}).
\end{align*}
All we have to check is that the second coalition of the losing part is indeed losing. To show this we note that $k_3-n_3<k_2$ and $k_3-n_2-n_3+1<k_2-n_2+1 \leq k_1$.  This shows that the second coalition of the losing part is indeed losing and proves the theorem.
\end{proof}

\section{Characterizations of Conjunctive Hierarchical Games}

First we obtain a structural characterization of conjunctive hierarchical
games.

\begin{theorem}
\label{shift-minimal-winning}
The class of conjuctive hierarchical simple games is exactly the class of complete
games with a unique shift-minimal winning coalition.
\end{theorem}

\begin{proof}
Let $H_\forall({\bf n},{\bf k})$ be a conjunctive hierarchical
game. By Theorem~\ref{all-exists-dual}, the dual game of
$H_\forall({\bf n},{\bf k})$ is a disjunctive hierarchical game
$H_\exists({\bf n},{\bf k}^{*})$. If we can prove that the class of
complete games with a unique shift-minimal winning coalition is
dual to the class of complete games with a unique shift-maximal
losing coalition, then by Theorem~\ref{shift-maximal-losing} this will be sufficient.

Let $G=(P,W)$ be a simple game with the unique shift-maximal
losing coalition $S$. By definition, $S^c$ is winning in $G^{*}.$
Let us prove that it is a shift-minimal winning coalition.
Consider any other coalition $X$ that can be obtained from $S^c$
by a shift in $G^{*}$. It means that there are players $i \in X$ and $j
\notin X$ such that $j \prec_{G^{*}} i$ and $X=( S^c \setminus
\{i\}) \cup \{j\}$. The complement of $X$ is the set $X^c=(S
\setminus \{j\})\cup \{i\}$. Furthermore, $j \prec_G i$. The
coalition $X^c$ is winning in $G$, because there does not exist a
losing coalition from which $S$ can be obtained by a shift.
Therefore, $X$ is losing in $G^{*}$ and $S^c$ is shift-minimal. Consider now a subset $X$ of
$S^c$. The complement $X^c$ of $X$ is a superset of $S$. Hence,
$X^c$ is winning in $G$ and $X$ is losing in $G^{*}$. Thus, $S^c$
is the shift-minimal winning coalition in $G^{*}$.

We claim that $S^c$ is the unique shift-minimal winning coalition
in $G^{*}$. Assume, to the contrary, there is another
shift-minimal winning coalition $X$ in $G^{*}$. As we have seen
above $X^cS$ would be shift-maximal losing coalition in $G$ and it is different from $S$, a
contradiction.
\end{proof}

It is interesting that the class of complete games with a unique shift-minimal winning coalition was studied before \cite{RePEc:eee:ejores:v:188:y:2008:i:2:p:555-568},
without noticing that this class is actually the class of conjunctive hierarchical games.

\begin{theorem}
Let $G=H_\forall({\bf n},{\bf k}) $ be an $m$-level conjunctive
hierarchical simple game. Then $G $ is a weighted majority game iff one of
the following conditions is satisfied:
\begin{enumerate}
\item[(1)] $m=1$;
\item[(2)] $m=2$ and $k_2=k_1+1$;
\item[(3)] $m=2$ and $n_2= k_2-k_1+1$;
\item[(4)] $m\in \{2,3\}$ and $k_1=n_1$. When $m=3$, $G$ is weighted if
and only if the reduced game $H_\forall({\bf n},{\bf
k})^{\{1^{n_1}\}}=H_\forall({\bf n}',{\bf k}')$, where ${\bf n}'=(n_2,n_3)$ and
${\bf k}'=(k_2-k_1,k_3-k_1)$ falls under (2) or (3);
\item[(5)] $m\in \{2,3,4\}$,  $k_m = k_{m-1}$,  and the reduced game
$H_\forall^{\{m^{n_m}\}}({\bf n},{\bf k})=H_\forall({\bf n}',{\bf k}')$,
where ${\bf n}'=(n_1,\ldots, n_{m-1})$ and ${\bf k}'=(k_1,\ldots,
k_{m-1})$ falls under one of the (1) -- (4);
\end{enumerate}
\end{theorem}
\begin{proof}
The theorem straightforwardly follows from Theorem~\ref{BTW}, the
duality between conjunctive hierarchical games and disjunctive
hierarchical game and Proposition~\ref{reduced}.
\end{proof}

\section{Further Research}

An interesting question in relation to complete simple games is to find how quickly can dimension grow 
depending on the number of players (for general games this growth is exponential). Thus it will be of interest to find the dimension of disjunctive hierarchical games or get an upper bound for their 
dimension. It should be noted that  the dimension of conjunctive hierarchical games, as it follows from results of
\cite{RePEc:eee:ejores:v:188:y:2008:i:2:p:555-568}  and Theorem~\ref{shift-minimal-winning}, is rather well-understood and has linear growth. 



\bibliographystyle{apacite}
\bibliography{grypiotr2006}

\begin{thebibliography}{}

\bibitem[\protect\citeauthoryear{%
Beimel%
, Tassa%
\BCBL{}\ \BBA{} Weinreb%
}{%
Beimel%
\ \protect\BOthers{.}}{%
{\protect\APACyear{2008}}%
}]{%
beimel:360}%
\APACinsertmetastar{%
beimel:360}%
Beimel, A.%
, Tassa, T.%
\BCBL{}\ \BBA{} Weinreb, E.%
%
\unskip\
\newblock
\APACrefYearMonthDay{2008}{}{}.
\newblock
\BBOQ{}\APACrefatitle{Characterizing Ideal Weighted Threshold Secret
  Sharing}{Characterizing ideal weighted threshold secret sharing}.\BBCQ{}
\newblock
\APACjournalVolNumPages{SIAM Journal on Discrete Mathematics}{22}{1}{360-397}.
\PrintBackRefs{\CurrentBib}

\bibitem[\protect\citeauthoryear{%
E.~Brickell%
\ \BBA{} Davenport%
}{%
E.~Brickell%
\ \BBA{} Davenport%
}{%
{\protect\APACyear{1990}}%
}]{%
springerlink:10.1007/0-387-34805-0-25}%
\APACinsertmetastar{%
springerlink:10.1007/0-387-34805-0-25}%
Brickell, E.%
\BCBT{}\ \BBA{} Davenport, D.%
%
\unskip\
\newblock
\APACrefYearMonthDay{1990}{}{}.
\newblock
\BBOQ{}\APACrefatitle{On the Classification of Ideal Secret Sharing Schemes}{On
  the classification of ideal secret sharing schemes}.\BBCQ{}
\newblock
\BIn{} G.~Brassard\ (\BED), \APACrefbtitle{Advances in Cryptology Ñ CRYPTOÕ 89
  Proceedings}{Advances in cryptology Ñ cryptoÕ 89 proceedings}\ (\BVOL~435,
  \BPG~278-285).
\newblock
\APACaddressPublisher{}{Springer Berlin / Heidelberg}.
\PrintBackRefs{\CurrentBib}

\bibitem[\protect\citeauthoryear{%
E\BPBI F.~Brickell%
}{%
E\BPBI F.~Brickell%
}{%
{\protect\APACyear{1990}}%
}]{%
Brickell:1990:ISS:111563.111607}%
\APACinsertmetastar{%
Brickell:1990:ISS:111563.111607}%
Brickell, E\BPBI F.%
%
\unskip\
\newblock
\APACrefYearMonthDay{1990}{}{}.
\newblock
\BBOQ{}\APACrefatitle{Some ideal secret sharing schemes}{Some ideal secret
  sharing schemes}.\BBCQ{}
\newblock
\BIn{} \APACrefbtitle{Proceedings of the workshop on the theory and application
  of cryptographic techniques on Advances in cryptology}{Proceedings of the
  workshop on the theory and application of cryptographic techniques on
  advances in cryptology}\ (\BPGS\ 468--475).
\newblock
\APACaddressPublisher{New York, NY, USA}{Springer-Verlag New York, Inc.}
\PrintBackRefs{\CurrentBib}

\bibitem[\protect\citeauthoryear{%
Carreras%
\ \BBA{} Freixas%
}{%
Carreras%
\ \BBA{} Freixas%
}{%
{\protect\APACyear{1996}}%
}]{%
CF:j:complete}%
\APACinsertmetastar{%
CF:j:complete}%
Carreras, F.%
\BCBT{}\ \BBA{} Freixas, J.%
%
\unskip\
\newblock
\APACrefYearMonthDay{1996}{}{}.
\newblock
\BBOQ{}\APACrefatitle{Complete simple games}{Complete simple games}.\BBCQ{}
\newblock
\APACjournalVolNumPages{Mathematical Social Sciences}{32}{2}{139--155}.
\PrintBackRefs{\CurrentBib}

\bibitem[\protect\citeauthoryear{%
Farr\`{a}s%
\ \BBA{} Padr\'{o}%
}{%
Farr\`{a}s%
\ \BBA{} Padr\'{o}%
}{%
{\protect\APACyear{2010}}%
}]{%
padro:2010}%
\APACinsertmetastar{%
padro:2010}%
Farr\`{a}s, O.%
\BCBT{}\ \BBA{} Padr\'{o}, C.%
%
\unskip\
\newblock
\APACrefYearMonthDay{2010}{}{}.
\newblock
\BBOQ{}\APACrefatitle{Ideal Hierarchical Secret Sharing Schemes}{Ideal
  hierarchical secret sharing schemes}.\BBCQ{}
\newblock
\BIn{} D.~Micciancio\ (\BED), \APACrefbtitle{Theory of Cryptography}{Theory of
  cryptography}\ (\BVOL\ 5978, \BPG~219-236).
\newblock
\APACaddressPublisher{}{Springer Berlin / Heidelberg}.
\PrintBackRefs{\CurrentBib}

\bibitem[\protect\citeauthoryear{%
Freixas%
\ \BBA{} Molinero%
}{%
Freixas%
\ \BBA{} Molinero%
}{%
{\protect\APACyear{2009}}%
}]{%
FMDAM}%
\APACinsertmetastar{%
FMDAM}%
Freixas, J.%
\BCBT{}\ \BBA{} Molinero, X.%
%
\unskip\
\newblock
\APACrefYearMonthDay{2009}{}{}.
\newblock
\BBOQ{}\APACrefatitle{Simple games and weighted games: A theoretical and
  computational viepoint.}{Simple games and weighted games: A theoretical and
  computational viepoint.}\BBCQ{}
\newblock
\APACjournalVolNumPages{Discrete Applied Mathematics}{157}{}{1496--1508}.
\PrintBackRefs{\CurrentBib}

\bibitem[\protect\citeauthoryear{%
Freixas%
\ \BBA{} Puente%
}{%
Freixas%
\ \BBA{} Puente%
}{%
{\protect\APACyear{2008}}%
}]{%
RePEc:eee:ejores:v:188:y:2008:i:2:p:555-568}%
\APACinsertmetastar{%
RePEc:eee:ejores:v:188:y:2008:i:2:p:555-568}%
Freixas, J.%
\BCBT{}\ \BBA{} Puente, M\BPBI A.%
%
\unskip\
\newblock
\APACrefYearMonthDay{2008}{}{}.
\newblock
\BBOQ{}\APACrefatitle{Dimension of complete simple games with
  minimum}{Dimension of complete simple games with minimum}.\BBCQ{}
\newblock
\APACjournalVolNumPages{European Journal of Operational
  Research}{188}{2}{555-568}.
\PrintBackRefs{\CurrentBib}

\bibitem[\protect\citeauthoryear{%
Neumann%
\ \BBA{} Morgenstern%
}{%
Neumann%
\ \BBA{} Morgenstern%
}{%
{\protect\APACyear{1944}}%
}]{%
vNM:b:theoryofgames}%
\APACinsertmetastar{%
vNM:b:theoryofgames}%
Neumann, J. von%
\BCBT{}\ \BBA{} Morgenstern, O.%
%
\unskip\
\newblock
\APACrefYear{1944}.
\newblock
\APACrefbtitle{Theory of games and economic behavior}{Theory of games and
  economic behavior}.
\newblock
\APACaddressPublisher{}{Princeton University Press}.
\PrintBackRefs{\CurrentBib}

\bibitem[\protect\citeauthoryear{%
Shamir%
}{%
Shamir%
}{%
{\protect\APACyear{1979}}%
}]{%
shamir:1979}%
\APACinsertmetastar{%
shamir:1979}%
Shamir, A.%
%
\unskip\
\newblock
\APACrefYearMonthDay{1979}{}{}.
\newblock
\BBOQ{}\APACrefatitle{How to share a secret}{How to share a secret}.\BBCQ{}
\newblock
\APACjournalVolNumPages{Commun. ACM}{22}{}{612--613}.
\PrintBackRefs{\CurrentBib}

\bibitem[\protect\citeauthoryear{%
Simmons%
}{%
Simmons%
}{%
{\protect\APACyear{1990}}%
}]{%
Simmons:1988}%
\APACinsertmetastar{%
Simmons:1988}%
Simmons, G\BPBI J.%
%
\unskip\
\newblock
\APACrefYearMonthDay{1990}{}{}.
\newblock
\BBOQ{}\APACrefatitle{How to (Really) Share a Secret}{How to (really) share a
  secret}.\BBCQ{}
\newblock
\BIn{} \APACrefbtitle{Proceedings of the 8th Annual International Cryptology
  Conference on Advances in Cryptology}{Proceedings of the 8th annual
  international cryptology conference on advances in cryptology}\ (\BPGS\
  390--448).
\newblock
\APACaddressPublisher{London, UK}{Springer-Verlag}.
\PrintBackRefs{\CurrentBib}

\bibitem[\protect\citeauthoryear{%
Stinson%
}{%
Stinson%
}{%
{\protect\APACyear{1992}}%
}]{%
Stinson:1992}%
\APACinsertmetastar{%
Stinson:1992}%
Stinson, D\BPBI R.%
%
\unskip\
\newblock
\APACrefYearMonthDay{1992}{}{}.
\newblock
\BBOQ{}\APACrefatitle{An explication of secret sharing schemes}{An explication
  of secret sharing schemes}.\BBCQ{}
\newblock
\APACjournalVolNumPages{Des. Codes Cryptography}{2}{}{357--390}.
\PrintBackRefs{\CurrentBib}

\bibitem[\protect\citeauthoryear{%
Tassa%
}{%
Tassa%
}{%
{\protect\APACyear{2007}}%
}]{%
Tassa:2007}%
\APACinsertmetastar{%
Tassa:2007}%
Tassa, T.%
%
\unskip\
\newblock
\APACrefYearMonthDay{2007}{}{}.
\newblock
\BBOQ{}\APACrefatitle{Hierarchical Threshold Secret Sharing}{Hierarchical
  threshold secret sharing}.\BBCQ{}
\newblock
\APACjournalVolNumPages{J. Cryptol.}{20}{}{237--264}.
\PrintBackRefs{\CurrentBib}

\bibitem[\protect\citeauthoryear{%
Taylor%
\ \BBA{} Zwicker%
}{%
Taylor%
\ \BBA{} Zwicker%
}{%
{\protect\APACyear{1999}}%
}]{%
tz:b:simplegames}%
\APACinsertmetastar{%
tz:b:simplegames}%
Taylor, A.%
\BCBT{}\ \BBA{} Zwicker, W.%
%
\unskip\
\newblock
\APACrefYear{1999}.
\newblock
\APACrefbtitle{Simple games}{Simple games}.
\newblock
\APACaddressPublisher{}{Princeton University Press}.
\PrintBackRefs{\CurrentBib}

\end{thebibliography}


\end{document}